\documentclass[11pt]{amsart}

\usepackage{fancybox,color}

\newcommand{\kom}[1]{}

\renewcommand{\kom}[1]{{\bf [#1]}}

\definecolor{gruen}{cmyk}{1.0,0.2,0.7,0.07}

\usepackage{latexsym}
\usepackage{amssymb}
\usepackage{mathrsfs}
\usepackage{amsmath}
\usepackage{fancybox,color,graphicx}
\usepackage{enumerate}

\usepackage[active]{srcltx}
\usepackage[colorlinks]{hyperref}
\usepackage{setspace}

 \def\1{\raisebox{2pt}{\rm{$\chi$}}}
 
\def\vint_#1{\mathchoice%
          {\mathop{\kern 0.2em\vrule width 0.6em height 0.69678ex depth -0.58065ex
                  \kern -0.8em \intop}\nolimits_{\kern -0.4em#1}}%
          {\mathop{\kern 0.1em\vrule width 0.5em height 0.69678ex depth -0.60387ex
                  \kern -0.6em \intop}\nolimits_{#1}}%
          {\mathop{\kern 0.1em\vrule width 0.5em height 0.69678ex
              depth -0.60387ex
                  \kern -0.6em \intop}\nolimits_{#1}}%
          {\mathop{\kern 0.1em\vrule width 0.5em height 0.69678ex depth -0.60387ex
                  \kern -0.6em \intop}\nolimits_{#1}}}
\def\vintslides_#1{\mathchoice%
          {\mathop{\kern 0.1em\vrule width 0.5em height 0.697ex depth -0.581ex
                  \kern -0.6em \intop}\nolimits_{\kern -0.4em#1}}%
          {\mathop{\kern 0.1em\vrule width 0.3em height 0.697ex depth -0.604ex
                  \kern -0.4em \intop}\nolimits_{#1}}%
          {\mathop{\kern 0.1em\vrule width 0.3em height 0.697ex depth -0.604ex
                  \kern -0.4em \intop}\nolimits_{#1}}%
          {\mathop{\kern 0.1em\vrule width 0.3em height 0.697ex depth -0.604ex
                  \kern -0.4em \intop}\nolimits_{#1}}}

\newcommand{\aveint}[2]{\mathchoice%
          {\mathop{\kern 0.2em\vrule width 0.6em height 0.69678ex depth -0.58065ex
                  \kern -0.8em \intop}\nolimits_{\kern -0.45em#1}^{#2}}%
          {\mathop{\kern 0.1em\vrule width 0.5em height 0.69678ex depth -0.60387ex
                  \kern -0.6em \intop}\nolimits_{#1}^{#2}}%
          {\mathop{\kern 0.1em\vrule width 0.5em height 0.69678ex depth -0.60387ex
                  \kern -0.6em \intop}\nolimits_{#1}^{#2}}%
          {\mathop{\kern 0.1em\vrule width 0.5em height 0.69678ex depth -0.60387ex
                  \kern -0.6em \intop}\nolimits_{#1}^{#2}}}

\newtheorem{theorem}{Theorem}[section]

\newtheorem{lemma}[theorem]{Lemma}
\newtheorem{proposition}[theorem]{Proposition}
\newtheorem{definition}[theorem]{Definition}
\newtheorem{remark}[theorem]{Remark}

\newtheorem{example}[theorem]{Example}

\newtheorem{claim}{Claim}[section]

  \newcommand{\R}{\mathbb{R}}
\newcommand{\Rn}{\mathbb{R}^n}

\newcommand{\xintloo}[1]{\int\limits_{#1} \kern-18pt\raise4pt\hbox to7pt {\hrulefill}\ }   \numberwithin{equation}{section}

\newcommand{\ec}{\normalcolor}

\def\XXint#1#2#3{{\setbox0=\hbox{$#1{#2#3}{\int}$}
     \vcenter{\hbox{$#2#3$}}\kern-.5\wd0}}

\begin{document}
 \title[Estimates for the $\infty$-Laplacian ]{A priori estimates for the $\infty$-Laplacian relative to Vector Fields}
 
   \author[F. Ferrari]{Fausto Ferrari}
 \address{Fausto Ferrari
 \hfill\break\indent
 Dipartimento di Matematica
\hfill\break\indent
dell'Universit\`a di Bologna,
\hfill\break\indent Piazza di Porta S. Donato, 5,
\hfill\break\indent \ec  40126 Bologna, Italy
\hfill\break\indent \ec  
{\tt fausto.ferrari@unibo.it}}

  \author[J. Manfredi]{Juan J. Manfredi}
 \address{Juan J. Manfredi
 \hfill\break\indent
Department of Mathematics
\hfill\break\indent
University of Pittsburgh
\hfill\break\indent Pittsburgh, PA 15260, USA
\hfill\break\indent
{\tt manfredi@pitt.edu}}

  \date{\today}
\ec
\thanks{F. F. was partially supported by INDAM-GNAMPA 2020 project: {\it Metodi di viscosit\`a e applicazioni a problemi non lineari con debole ellitticit\`a.}}

\maketitle
\begin{center}
\textit{Dedicated to our
friend Bruno Franchi 
} 
\end{center}
\begin{abstract}
In this paper we prove a priori  H\"older and Lipschitz regularity estimates for viscosity solutions  equations governed by the inhomogeneous  infinite Laplace operator relative to a frame of vector fields.

\end{abstract}
{\singlespacing
\tableofcontents        
}
\section{Introduction}
The main results in this manuscript are the a priori local H\"older and Lipschtiz continuity of viscosity solutions to the problem
 \begin{equation}\label{mainequation0}
\sum_{i,j=1}^n X_i X_j u(x) X_iu(x) X_ju(x)    = f (x, u(x), X_1u(x), \ldots X_n u(x)),
   \end{equation}
where $f$ is a real valued continuous functions and $X_1, X_2, \ldots X_n$ are linearly independent  smooth vector fields in $\Rn$. \par
We write $$D_{\mathfrak{X}}u=\sum_{i=1}^n X_i u(x) X_i$$  for  the gradient of the function $u$ relative to the  frame of vector fields 
$\mathfrak{X}=\{X_1, X_2, \dots, X_n\}$.
We consider  $\mathbb{R}^n$ as a Riemannian manifold  with a metric induced by the frame   $\mathfrak{X}$. This frame determines a Riemannian metric $g$ by requiring that $\mathfrak{X}(x)=\{X_1(x), X_2(x), \dots, X_n(x)\}$ is an orthonormal basis for the metric $g_x$ in the tangent space to $\mathbb{R}^n$  at $x$ (which we identify with $\mathbb{R}^n$); that is, we have
$$ g_x(X_i(x), X_j(x))= \delta_{ij}\text{  for  } i, j=1\ldots n.$$ 
 Write \begin{equation}
 X_i(x)=\sum_{j=1}^n a_{ij}(x)\frac{\partial}{\partial x_j}
 \end{equation}
  for smooth functions $a_{ij}(x)$.  Denote by $\mathbb{A}(x)$ the matrix whose
$(i,j)$-entry is $a_{ij}(x)$. We always assume that $\det(\mathbb{A}(x))\not=0$.  Let $\mathbb{G}(x)$ denote the matrix of $g_x$ with respect to the Euclidean coordinates. We then have
\begin{equation}\label{metric} \mathbb{G}(x)= \left(\mathbb{A}^t(x) \mathbb{A}(x)\right)^{-1}.
\end{equation}

  We can write equation \eqref{mainequation0} as 
\begin{equation}\label{inftyharmonic0}
 \Delta_{\mathfrak{X},\infty} u =\left\langle \left(D_\mathfrak{X}^2 u\right)^* D_\mathfrak{X}u, D_\mathfrak{X}u \right\rangle_{g}=f(x,u(x), D_{\mathfrak{X}}u(x))
 \end{equation}
 is the $\infty$-Laplacian relative to the frame  $\mathfrak{X}$, where $g$ is the Riemannian metric determined by $\mathfrak{X}$.
\par

We use the notation $d(x,y)$ for  the Riemannian distance determined by $g$. For a point $x\in \mathbb{R}^n$ the injectivity radius is $i (x)>0$. The metric ball centered at $x$ with radius $r>0$ is denoted by $B_r(x)$.
The gradient  of a smooth function $u\colon\mathbb{R}^n\mapsto\mathbb{R}$ relative to $\mathfrak{X}$  agrees with the Riemannian gradient of the function $u$ (see Lemma \ref{grads} below). 
The $\mathfrak{X}$-second derivative matrix $D_\mathfrak{X}^2 u$ is an $n\times n$ matrix, not necessarily symmetric, with entries $X_i(X_j(u))$. We will consider its symmetrization
$$\left(D_\mathfrak{X}^2 u\right)^*=\frac{ D_\mathfrak{X}^2 u+ (D_\mathfrak{X}^2 u)^t}{2}$$ and note that $\left(D_\mathfrak{X}^2 u\right)^*$ is, in general, different from $\text{Hess}(u)$ the Riemannian Hessian of the function $u$. See 
Example \ref{hessian} below. 
\par
Our starting point is the fact that 
the function
$u(x)=d(x_0,x)$,  which is  smooth in the set $B_{i(x_0)}(x_0)\setminus\{x_0\}$,  satisfies the eikonal equation
\begin{equation}\label{eikonal0}
| D_\mathfrak{X} u |_{g}=1, 
\end{equation}
and it is $\infty$-harmonic
\begin{equation}\label{inftyharmonic0}
 \Delta_{\mathfrak{X},\infty} u =\left\langle \left(D_\mathfrak{X}^2 u\right)^* D_\mathfrak{X}u, D_\mathfrak{X}u \right\rangle_{g}=0.
\end{equation}
See Proposition \ref{dproperties1} below. For more information  about distances and infinity-Laplacians see \cite{BDM09, B10}.
\par We shall also use the fact that $(x,y)\mapsto d^2(x,y)$ is locally smooth, Proposition \ref{dproperties2}. Thus,  functions of the distance are available as test functions for the viscosity formulation of \eqref{mainequation0} that we describe next.
 \begin{definition}\label{viscositytestfunctions0}
   An upper semi-continuous function $u$ is a viscosity subsolution of \eqref{inftyharmonic0} in a domain $\Omega\subset\mathbb{R}^n$  if whenever $\phi\in C^2(\Omega) $ touches $u$ from above at at a point $x_0\in \Omega$ we have 
   $$  \Delta_{\mathfrak{X},\infty} \phi(x_0) \ge f(x_0, \phi(x_{0}), D_{\mathfrak{X}}\phi(x_{0})).$$
 A lower semi-continuous function $v$ is a viscosity supersolution of \eqref{inftyharmonic0} in a domain $\Omega$  if whenever $\phi\in C^2(\Omega) $ touches $v$ from below at at a point $x_0\in \Omega$ we have 
   $$  \Delta_{\mathfrak{X},\infty} \phi(x_0) \le  f(x_0, \phi(x_{0}), D_{\mathfrak{X}}\phi(x_{0})).$$
   \end{definition}
   Recall that $\phi$ touches $u$ from above at $x_0$ means $\phi(x)\le   u(x)$ for all $x$ in a neighborhood of $x_0$ and $\phi(x_0)=u(x_0)$.  To define $\phi$ touches $u$ from  below at $x_0$ just reverse the inequality. \par
   A viscosity solution is both a super- and a subsolution.
  Our main results are  the following:
  \begin{theorem}\label{apriori0} Let $\Omega\subset\mathbb{R}^n$ be a domain and 
 $f\colon \Omega\mapsto\mathbb{R}$ be a  continuous function. Let  $u$ be a viscosity solution of the inhomogeneous $\infty$-Laplace equation
\begin{equation}\label{homogeneous}
  \Delta_{\mathfrak{X},\infty} u(x)=f (x)\text{   in  } \Omega
  \end{equation}Then, the function $u$ is locally Lipschitz continuous. More precisely, for all $x_0\in\Omega$ such that $B_{2\,i(x_0)}(x_0)\subset\Omega$ we have  
\begin{equation}\label{lip1}
|u(x)-u(y)|\le L \, d(x,y),  
\end{equation}  for $x,y\in B_{i(x_0)/4}(0)$,
where  $L$ depends only on $\|f \|_{L^\infty(B_{2\,i(0)})}$, $\|u\|_{L^\infty(B_{2\,i(0)})}$ and the infimum of the injectivity radius on the compact set $ \overline{B_{i(x_0)}(x_0)}$.

\end{theorem}
 \begin{theorem}\label{apriori1} Let $\Omega\subset\mathbb{R}^n$ be a domain and 
  $f\colon \Omega\times\mathbb{R}\times \mathbb{R}^{n}\mapsto\mathbb{R}$ be a  continuous function  satisfying the condition
\begin{equation}\label{gradproblem}
|f(x,p,\xi)| \le C_{2} |\xi|_g^{\beta}+C_{3},
\end{equation}
where $0\le \beta<4$  and $C_{2}, C_{3}$ are nonnegative constants. 
Let $u$ be a viscosity solution of the inhomogeneous $\infty$-Laplace equation
$$  \Delta_{\mathfrak{X},\infty} u(x)=f (x, u(x), D_{\mathfrak{X}} u(x) )\text{   in  } \Omega.$$ Then, the function $u$ is locally H\"older continuous with exponent
$\alpha<\min\{1, \frac{4-\beta}{3}\}$. More precisely, for $x_0\in \Omega$ such that  $B_{2\,i(x_0)}(x_0)\subset\Omega$ and  $x,y\in B_{i(x_0)/4}(x_0)$ we have
\begin{equation}\label{lip1}
|u(x)-u(y)|\le L_{1} \, d(x,y)^{\alpha},  
\end{equation}
where  $L_{1}$ depends only on $\|u \|_{L^\infty(B_{i(x_0)} (x_0))}$, the constants $C_2$ and $C_3$, the exponen $\beta$,  the injectivity radius $i(x_0)$, and a constant $C(\overline{B_{2\,i(x_0)}(x_0)},g)$ depending only on the metric $g$ and the compact set $ \overline{B_{2\,i(x_0)}(x_0)}$.

\end{theorem}

In the Euclidean case, where $X_i=\partial_{x_i}$,  Theorem \ref{apriori0} was proven by Lindgren in \cite{L14}. 
In the Riemannian  case,  Theorem \ref{apriori0} was proven by Lu, Miao, and Zhu  in \cite{LMZ19}.
They consider the equation
\begin{equation}\label{lmz1}
\langle D\langle A(x)Du(x), Du(x)\rangle , A(x) Du(x) \rangle= f(x),
\end{equation}
where $A\in C^{1}$ and $f$ is continuous. 
Their proof is based on  using the Hamilton-Jacobi equation  $\langle A(x), p\rangle+\lambda u =1$ to approximate the intrinsic metric associated to $A(x)$.  It turns out  that equations \eqref{lmz1} and \eqref{homogeneous} are the same equation since we have
$$\left\langle \left(D_\mathfrak{X}^2 u\right)^* D_\mathfrak{X}u, D_\mathfrak{X}u \right\rangle_{g}
= \left\langle D\left\langle A(x)Du(x), Du(x)\right\rangle , A(x) Du(x) \right\rangle$$ 
when we take $A(x)=  \mathbb{A}^t(x) \mathbb{A}(x)$. Note that $A(x)=\mathbb{G}(x)^{-1}$, where $\mathbb{G}(x)$ is the matrix of the metric $g_x$, see \eqref{metric}. \par
Our proof of Theorem \ref{apriori0} follows by using directly  properties of the Riemannian metric  that we discuss in Section \S \ref{three} below. In particular we establish the analog of the Euclidean formula $\Delta_{\infty} |x|^{\alpha}= 4 \alpha^3 (\alpha-1)|x|^{3\alpha-4}$ for a general Riemannian metric
$$ \langle (D_\mathfrak{X}^2 d^{\alpha} )^* \cdot D_\mathfrak{X}d^{\alpha}, D_\mathfrak{X}d^{\alpha}\rangle_g= 4 \alpha^3 (\alpha-1) d^{3\alpha -4}$$ whenever $x\mapsto d(x,y)$ is smooth, see Lemma \ref{dalpha} below.
Another important result in \cite{LMZ19} is the everywhere differentiability of the solutions when $f\in C^1$. In the Euclidean case Lindgren \cite{L14} extended the result of Evans and Smart \cite{ES11b} to the non-homogeneous case by establishing an almost-monotonicity property of incremental quotients to obtain the linear approximation property and the everywhere differentiability. 
In the Riemanniann case Lu, Miao and Zhou again use Hamilton-Jacobi equations to establish their result. 

\par

Our proof of Theorem \ref{apriori1}  is an adaptation of the standard penalization argument  with several challenges posed by the non-commutativity of the vector fields in the frame. This is the Crandall-Ishii-Lions method for  regularity of viscosity solutions (see for example  \cite{IL90, Cr97}). The authors found particularly useful the reading of \cite{IS1} and \cite{IS13} as well.   About such approach, there are many contributions in literature. Among them, we wish to recall the following works \cite{BGI}, \cite{BGL}, \cite{FV}, \cite{FG21}, where the regularity of viscosity solutions of truncated operators has been studied. Moreover, always in the frame of a degenerate situation, but in a non-commutative structures, we point out the results contained in \cite{F}, \cite{FVe} and \cite{G}.   
 We develop several properties of the second derivatives of the metric in Section \S \ref{three}  to double the variables and use an adapted theorem of sums.  A key estimate is a bound for the symmetrized second derivatives of the distance, Lemma \ref{symmetrizedbound} below, that we obtain from the eikonal equation. Note that we allow for a general first order term 
$f(x,u, D_\mathfrak{X}u)$ but that we only get H\"older estimates. 
\par
In addition to the blow-up and duality estimates in the homogeneous case in \cite{ES11} and \cite{ES11b}, we would like to mention  \cite{LW08},  where the inhomogeneous $\infty$-Laplacian was treated from the PDE point of view, \cite{AS12} for a finite differences treatment,  and \cite{PSSW09} for a tug-of-war interpretation. 
Sharp  estimates for the Sobolev derivative of $|\nabla u|^\alpha$ for  solutions of \eqref{homogeneous} in the Euclidean plane $\R^2$ are  obtained in \cite{KZZ19} when $f$ is continuous, non vanishing,  and of bounded variation.

\par\normalcolor
A representative example is the Riemannian Heisenberg group,  where the frame $\mathfrak{X}=\{X, Y, Z\}$ is given by  
 the left invariant vector fields in $\mathbb{R}^3$  with respect to the Heisenberg group operation
 $(x,y,z)*(x',y',z')=(x+x', y+y', z+z'+\frac12 (xy'-y'x))$. These vector fields are  $X= \partial_x-\frac{1}{2} y \, \partial_z$,  $Y =\partial_y+\frac{1}{2} x \, \partial_z$, and 
 $Z=\partial_z$.  The Levi-Civita connection (computed in Chapter 2 of \cite{CDPT07}) is determined by the equations
 \begin{equation*}\begin{array}{lcccccr}
 \nabla_X X & =& \nabla_Y Y& =& \nabla_Z Z& =0,&  \\
  \nabla_X Y & =& \frac{1}{2}Z, & &   \nabla_Y X & =& -\frac{1}{2} Z , \\
  \nabla_Z X& =& \nabla_X Z&=& -\frac{1}{2} Y,  & \text{and} & \\
  \nabla_Z Y &=& \nabla_Y Z& =& \frac{1}{2} X.  & & 
 \end{array}
 \end{equation*} The matrix of $Hess(u)$ with respect to basis $\{X,Y,Z\}$ is then
 \begin{equation*}
\left( \begin{array}{ccc}
XXu & XYu-\frac12 Zu & XZu +\frac12 Yu\\
YXu +\frac12 Zu & YYu & YZu-\frac12 Xu\\
ZXu+\frac12 Yu&  ZYu-\frac12 Xu & ZZu
 \end{array}\right),
 \end{equation*}
 which differs from $\left(D_\mathfrak{X}^2 u\right)^*$ in the $(1,3), (2,3), (3,1)$ and $(3,2)$ entries.\par
 Nevertheless, in this particular case we still have that the Riemannian $\infty$-Laplacian
 \begin{equation}\label{riemannianinfinity}
 \Delta_{g,\infty} u = \langle Hess(u) \mathfrak{X}u, \mathfrak{X} u\rangle_g \end{equation}
 agrees with the
 frame $\infty$-Laplacian
 $$\Delta_{\mathfrak{X}, \infty}u=  \langle \left(D_\mathfrak{X}^2 u\right)^* \mathfrak{X}u, \mathfrak{X}u\rangle_g,
 $$
 as a direct calculation shows. Therefore, Theorems \ref{apriori0} and \ref{apriori1}  also hold for \eqref{riemannianinfinity} in the Riemannian Heisenberg case.
 \par
  
  The plan of the paper is as follows: in Section \S 2 we present the details of our set-up. In section \S 3  we present the proof of bound for the symmetrized second derivatives of the distance. Some facts about  viscosity solutions and frames are in Section \S4. The proof of the main results Theorems \ref{apriori0}  and \ref{apriori1} are in Sections \S 5 and \S 6 respectively.\par
  
 \textsc{Acknowledgement:} We thank the anonymous referee for bringing to our attention the reference \cite{LMZ19} and 
 for several suggestions that have improved the readability of the paper.

\section{Preliminaires}
In $\mathbb{R}^n$ the function  $u(x)=|x-x_0|$  satisfies both the eikonal equation $|\nabla u|=1$ and the $\infty$-Laplace equation $\Delta_{\infty}(u)= \langle D^2u\, \nabla u, \nabla u \rangle=0$ in $\mathbb{R}^n\setminus\{x_0\}$. 
A similar phenomena occurs for the case of Riemannian and sub-Riemannian manifolds, where  the function $u(x)=d(x,x_0)$  satisfies the eikonal equation and the infinity-Laplace equation whenever it is smooth, 
see Proposition \ref{dproperties1} below.

We consider the case where the manifold is $\mathbb{R}^n$ endowed with a Riemannian metric induced by a frame
 $\mathfrak{X}=\{X_1, X_2, \dots, X_n\}$; that is, $\mathfrak{X}$ is a collection of $n$ linearly independent vector fields in $\mathbb{R}^n$.  \par
We first write down an appropriate Taylor theorem adapted to the frame $\mathfrak{X}$. For this, we will use exponential
coordinates as done  in \cite{NSW85}.  Fix a point $p\in \mathbb{R}^n$ and let
$t=(t_1, t_2,\ldots, t_n)$ denote a vector close to zero. We define the flow exponential based
at $p$ of $t$, denoted by $\Theta_p(t)$, as follows.
Let $\gamma$ be the unique solution to the system  of ordinary differential equations
$$\gamma'(s)=\sum_{i=1}^n t_i X_i(\gamma(s))$$ satisfying the initial condition
$\gamma(0)=p$. We set  $\Theta_p(t)=\gamma(1)$ and note this is defined in 
a neighborhood of zero. \par
Applying the one-dimensional  Taylor's formula to 
$u(\gamma(s))$ we get 
\begin{lemma}\label{taylor}(\cite{NSW85}) 
Let $u$ be a smooth function in a neighborhood of $p$. We have:
$$u\left(\Theta_p(t)\right)=
u(p)+\langle D_{\mathfrak{X}}u(p)  ,t  \rangle+
\frac{1}{2}\langle   \left(D_\mathfrak{X}^2 u(p) \right)^*  t , t  \rangle
+o(|t|^2)$$
as $t\to 0$.\par
\end{lemma} 
If instead of the flow exponential based at $p$ we use the Riemannian exponential $\text{Exp}_p(t)$ we have
\begin{lemma}\label{riemmanniantaylor}
Let $u$ be a smooth function in a neighborhood of $p$. We have:
$$u\left(\text{Exp}_p(t)\right)=
u(p)+\langle D_{\mathfrak{X}}u(p)  ,t  \rangle+
\frac{1}{2}\langle   \text{Hess}( u)(p)  t , t  \rangle
+o(|t|^2)$$
as $t\to 0$.\par
\end{lemma} 
For the proof, see for example Chapter 8 in \cite{GQ20}. 
Applying Lemma \ref{taylor} to the coordinate functions we obtain:
\begin{lemma}\label{lemma3} Write  $\Theta_p(t)=\left( \Theta^1_p(t),\Theta^2_p(t),\dots,\Theta^n_p(t)     \right)$. Note that we can think of $X_i(x)$ as the $i$-th row of $\mathbb{A}(x)$.
Similarly  $D\Theta^k_p(0)$ is the $k$-column of $\mathbb{A}(p)$ so that 
$$D\Theta_p(0)=\mathbb{A}(p).$$ 
 \end{lemma}
In particular, the mapping $t\mapsto \Theta_p(t)$ is a diffeomorphism taking  
a neighborhood of $0$ into a neighborhood of $p$.\par
For vector fields  $Y=\sum_{i=1}^n y_i X_i$ and $Z=\sum_{i=1}^n z_i X_i$ we  have
$$ \langle Y, Z\rangle_g= \sum_{i=1}^n y_i z_i.$$ Writing $X$ and $Y$ in Euclidean coordinates
$Y=\sum_{i=1}^n \bar{y}_i \partial_{x_i}$ and $Z=\sum_{i=1}^n \bar{z}_i \partial_{x_j}$  we get
$$\langle Y, Z\rangle_g=\sum_{i,j=1}^n  \bar{y}_i\bar{z}_i\langle \partial_{x_i}, \partial_{x_i}\rangle_g=\sum_{i,j=1}^n  \bar{y}_i\bar{z}_i \mathbb{G}_{ij}
= \sum_{i,j=1}^n  \bar{y}_i\bar{z}_j  \left(\mathbb{A}^t \mathbb{A}\right)^{-1}_{ij}. $$ Conclude that
$$\langle Y, Z\rangle_g= \langle  \left(\mathbb{A}^t \mathbb{A}\right)^{-1} Y, Z\rangle= \langle
(\mathbb{A}^{-1})^t Y,(\mathbb{A}^{-1})^t  Z\rangle $$ and $$\langle \mathbb{A}^ t Y, \mathbb{A}^t Z\rangle_g= \langle Y, Z \rangle.$$
\begin{lemma}\label{grads}
Let $u\colon\mathbb{R}^n\mapsto\mathbb{R}$ be a smooth function. Then,  the Riemannian gradient of $u$ relative to the metric $g$ is the vector field $D_\mathfrak{X}u=\sum_{j=1}^n X_j(u) X_j
$ with length
$$| D_\mathfrak{X}u|_g= \langle D_\mathfrak{X}u, D_\mathfrak{X}u\rangle^{1/2}_g= \left(\sum_{i=1}^n (X_i u)^2\right)^{1/2}.$$

\end{lemma}
\begin{proof}
The Riemannian gradient is the vector field is give by the expression
$$\sum_{i,j=1}^n \mathbb{G}^{ij}\frac{\partial f}{\partial x_i}\frac{\partial}{\partial x_i}= \sum_{i,j=1}^n (\mathbb{A}^t \mathbb{A})_{ij} \frac{\partial f}{\partial x_i}\frac{\partial}{\partial x_i}= \sum_{j=1}^n X_j(f) X_j. $$ 
\end{proof}\
\begin{example}\label{hessian}
Note that it is not true, in general, that the Riemannian Hessian of a function $u$ given by $Hess(u)(V,W)= VWu-\nabla_V W u$, where $V$ and $W$ are arbitrary vector fields,  equals the symmetrized  second derivatives relative to the frame  $\left(D_\mathfrak{X}^2 u\right)^*$. Here  $\nabla$ denotes the Levi-Civita connection.  
Consider the  Riemannian Heisenberg group $\mathbb{H}$ with 
 left invariant vector fields in $\mathbb{R}^3$ given by $X= \partial_x-\frac{1}{2} y \, \partial_z$,  $Y =\partial_y+\frac{1}{2} x \, \partial_z$, and 
 $Z=\partial_z$. 
 The Levi-Civita connection (computed in Chapter 2 of \cite{CDPT07}) is determined by the equations
 \begin{equation*}\begin{array}{lcccccr}
 \nabla_X X & =& \nabla_Y Y& =& \nabla_Z Z& =0,&  \\
  \nabla_X Y & =& \frac{1}{2}Z, & &   \nabla_Y X & =& -\frac{1}{2} Z , \\
  \nabla_Z X& =& \nabla_X Z&=& -\frac{1}{2} Y,  & \text{and} & \\
  \nabla_Z Y &=& \nabla_Y Z& =& \frac{1}{2} X.  & & 
 \end{array}
 \end{equation*} The matrix of $Hess(u)$ with respect to basis $\{X,Y,Z\}$ is then
 \begin{equation*}
\left( \begin{array}{ccc}
XXu & XYu-\frac12 Zu & XZu +\frac12 Yu\\
YXu +\frac12 Zu & YYu & YZu-\frac12 Xu\\
ZXu+\frac12 Yu&  ZYu-\frac12 Xu & ZZu
 \end{array}\right),
 \end{equation*}
 which differs from $\left(D_\mathfrak{X}^2 u\right)^*$ in the $(1,3), (2,3), (3,1)$ and $(3,2)$ entries.\par

 \end{example}
 \begin{remark}
\par
The mapping $$t\mapsto \Theta_p(t)$$ is the flow exponential that agrees with the Lie group exponential when the frame $\mathfrak{X}$ happens to be a basis for a Lie algebra of an $n$-dimensional Lie group. \par
Associated to the Riemannian metric $g$ we also have the Riemannian exponential $t\mapsto Exp_p(t)$ defined using geodesics. Both are diffeomorphisms in a neighborhood of $0$. Lemma \ref{grads} shows that they agree up to first order since the Riemannian gradient equals the frame gradient (the linear terms in the Taylor development are the same). \par
Note that for the Riemannian Heisenberg group the flow exponential mapping is  the group multiplication
$$\Theta_p(t)=p\cdot \Theta_0(t)=(x+t_1,y+t_2,z+t_3+(1/2)(xt_2-yt_1)).$$
Taking into account the explicit formula for the Riemannian  exponential $Exp_p(t)$ in the Riemannian Heisenberg group (see \cite{BN16}) we conclude that $ \Theta_p(t)$ and $Exp_p(t)$ are different mappings.\par
On the other hand, the flow exponential agrees with the Riemannian exponential in the case of Lie groups equipped with a bi-invariant metric,  see Chapter 21 in \cite{GQ20} or Chapter 2 in \cite{AB15}. Compact Lie groups, like $SO(n)$, admit a bi-invariant metric. In fact a connected  Lie group admits a bi-invariant metric if and only if it is isomorphic to the product of a compact group and an abelian group (Lemma 7.5 in \cite{M76}.)
\end{remark} 

\begin{proposition}\label{dproperties1}
Fix $x_0\in \mathbb{R}^n$ and consider the function
$u(x)=d(x_0,x)$. This function is smooth in the set $B_{i(x_0)}(x_0)\setminus\{x_0\}$, it satisfies the eikonal equation
\begin{equation}\label{eikonal}
| D_\mathfrak{X} u |_{g}=1, 
\end{equation}
and it is $\infty$-harmonic
\begin{equation}\label{inftyharmonic}
 \Delta_{\mathfrak{X},\infty} u =\left\langle \left(D_\mathfrak{X}^2 u\right)^* D_\mathfrak{X}u, D_\mathfrak{X}u \right\rangle_{g}=0
\end{equation}
\end{proposition}
\begin{proof} Recall that $d(x,x_0)$ is smooth in $B_{i(x_0)}(x_0)\setminus\{x_0\}$ (see Chapter 6 in \cite{L18} for example). 
The fact that $d(x_0,x)$ satisfies \eqref{eikonal} and \eqref{inftyharmonic}  is also well-known (see Corollary 4.12 in \cite{DMV13}).
\end{proof}

\begin{proposition}\label{dproperties2} Fix $x_0\in \mathbb{R}^n$. The function  $(x,y)\mapsto d^2(x,y)$  is smooth in  $B_{i(x_0)}(x_0)\times B_{i(x_0)}(x_0)$.
\end{proposition}
     
\begin{proof} See Chapter 6 in \cite{L18}. 
\end{proof}
We conclude that  given a compact subset $K\subset\mathbb{R}^n$, there exists a  constant $C_0(K)>0$ such that the function $v_{y}(x)= d^2(x,y)$ satisfies
\begin{equation}
|D^2_\mathfrak{X}v_y(x)|_g\le C_0(K),
\end{equation}
whenever $x,y \in B_{i(x_0)/2}(x_0)$ for all $x_0\in K$.

\begin{proposition}\label{boundford2} Given a compact subset $K\subset\mathbb{R}^n$, there exists a  constant $C_1(K)>0$ such that the function $u_{y}(x)= d(x,y) $ satisfies
\begin{equation}
   |D^2_\mathfrak{X}u_y(x)|_g\le C_1(K) \frac{1}{d(x,y)},
   \end{equation}
   whenever $x,y \in B_{i(x_0)/2}(x_0)$ for all $x_0\in K$.
\end{proposition}

\begin{proof}
For $y\in B_{i(x_0)/2}(x_0)$  and  
  $x\not=y$ we have
$$\begin{array}{rcl}
 X_i(x) \left(X_j(x) \, d^2(x,y)\right)& = & X_i(x) \left( 2 d(x,y) \, X_j(x)(d(x,y)  )    \right)\\
 &=& 2 X_i(x)(d(x,y))\, X_j(x)(d(x,y))+ 2\, d(x,y) X_i(x)(X_j(x)(d(x,y)),
 \end{array} $$ from which we deduce
 $$
 \begin{array}{crl}
 |X_i(x)(X_j(x)(d(x,y))| & \le &  \frac{C_0}{2\, d(x,y)} +  \frac{ | X_i(x)(d(x,y))| \, |X_j(x)(d(x,y)|}{ d(x,y)} \\
  & \le &  \frac{C_0/2+1}{d(x,y)}.
 \end{array}
 $$
 We can then take $ C_1(K)= n^2(C_0(K)/2+1)$.
 \end{proof}\par

 \section{Second Derivatives of the Metric}\label{three}
 \par
  In this section we work in a region where the function of two variables $(x,y)\mapsto d(x,y)$ is smooth. This is the case when $x$ and $y$ are in the ball $ B_{i(z)}(z)$ for some point $z$ and $x\not=y$. 
  Our starting point is that for fixed $y$ the function $x\mapsto d(x,y)$  satisfies the eikonal equation in a punctured neighborghood of $y$ 
  \begin{equation}
  \label{eikonal1} \sum_{i=1}^n (X_i^x d(x,y))^2=1,\end{equation}
  where we have written $X_i^x$ to indicate that the vector field $X_i$ is acting on the $x$ variable. See Proposition \ref{dproperties1} above.
  Similarly, for a fixed $x$ the function $y\mapsto d(x,y)$  satisfies the eikonal equation in a punctured neighborghood of $x$ 
  \begin{equation}
  \label{eikonal2} \sum_{i=1}^n (X_i^y d(x,y))^2=1,\end{equation}
  where we have written $X_i^y$ to indicate that the vector field $X_i$ is acting on the $y$ variable.
  Next we apply $X_j^x$ and $X_j^y$ to both \eqref{eikonal1} and \eqref{eikonal2}  obtaining the following result whose proof is a straightforward computation.
  \begin{lemma}\label{goodeikonal} For $j=1,\ldots, n$ we have
  $$
  \begin{array}{rclrcl}
  \displaystyle\sum_{i=1}^n X_i^x d\,  X_j^x X_i^x d= 0 &, &   \displaystyle\sum_{i=1}^n X_i^y d\,  X_j^x X_i^y d=0, \\
  \displaystyle\sum_{i=1}^n X_i^x d\,  X_j^y X_i^x d= 0 &, &   \displaystyle\sum_{i=1}^n X_i^y d\,  X_j^y X_i^y d=0.   \end{array} 
  $$
  \end{lemma}
  We introduce the following $n\times n$ matrices of second derivatives:
  $$ \begin{array}{rcl}
  (D_\mathfrak{X}^{2,x}u )_{ij} =  X_i^x X_j^x u   &, &   (D_\mathfrak{X}^{2,x,y}u )_{ij} =   X_i^x X_j^y u  \\
    (D_\mathfrak{X}^{2,y,x}u )_{ij} =   X_i^y X_j^x u   &, &   (D_\mathfrak{X}^{2,y}u )_{ij} =   X_i^y X_j^y u.     \end{array}
  $$
  With this notation, recalling Lemma \ref{goodeikonal}, we obtain
  \begin{equation}\label{matrixeikonal}
  \begin{array}{rcl}
  D_\mathfrak{X}^{2,x}d \cdot D_\mathfrak{X}^x d =0 & ,&  D_\mathfrak{X}^{2,x,y}d \cdot D_\mathfrak{X}^y d=0 ,\\
    D_\mathfrak{X}^{2,y,x}d \cdot D_\mathfrak{X}^x d =0 & ,&  D_\mathfrak{X}^{2,y,y}d \cdot D_\mathfrak{X}^y d=0.
   \end{array}
 \end{equation}
To keep the notation simpler we also denote by $\mathfrak{Z}$ the frame  $\mathfrak{X}\otimes\mathfrak{X}$ in $\mathbb{R}^n\times\mathbb{R}^n$ obtaining by 
 considering two copies of $\mathfrak{X}$.  

  The $\mathfrak{Z}$-gradient of a function $u(x,y)$ in the variables $(x,y)$ is the $2n \times 1$ vector field
 $$D_\mathfrak{Z}u= \left(\begin{array}{c} D_\mathfrak{X}^x u \\
 D_\mathfrak{X}^y  u \end{array}\right).
 $$
 Note that $|D_\mathfrak{Z}d |_g= \sqrt{ |D^x_\mathfrak{X}d |_g^2+  |D^y_\mathfrak{X}d |_g^2} = \sqrt{2}$.
The  second derivative of $u(x,y)$  is given by the $2n \times 2n$ matrix
 $$D^2_\mathfrak{Z}u= \left(\begin{array}{cc} D^{2,x}_\mathfrak{X} u & D^{2,x,y}_\mathfrak{X} u  \\
D^{2,y,x}_\mathfrak{X} u & D^{2,y}_\mathfrak{X} u  \end{array}\right).
 $$ 
 From the identities \eqref{matrixeikonal} it follows  that
 \begin{equation}\label{eikonalbest}
 D_\mathfrak{Z}^2 d \cdot D_\mathfrak{Z}d=0
 \end{equation} and,  similarly for the symmetrized second derivatives, we obtain
  \begin{equation}\label{eikonalbest2}
 \langle (D_\mathfrak{Z}^2 d)^*\cdot D_\mathfrak{Z}d, D_\mathfrak{Z}d\rangle_g =0.
 \end{equation} 
 Since we have $D_\mathfrak{Z} d^\alpha= \alpha d^{\alpha-1} D_\mathfrak{Z} d$ and 
 $D^2_\mathfrak{Z} d^\alpha= \alpha d^{\alpha-1} D_\mathfrak{Z}^2 d+ \alpha (\alpha -1) d^{\alpha-2} (D_\mathfrak{Z} d
 \otimes D_\mathfrak{Z}d)$ we get
 $$\langle D_\mathfrak{Z}^2 d^{\alpha} \cdot D_\mathfrak{Z}d^{\alpha}, D_\mathfrak{Z}d^{\alpha}\rangle_g =
 \alpha^3 (\alpha-1) d^{3\alpha -4} \langle (D_\mathfrak{Z} d
 \otimes D_\mathfrak{Z}d)\cdot  D_\mathfrak{Z} d, D_\mathfrak{Z} d\rangle_g$$ and
  $\langle (D_\mathfrak{Z} d
 \otimes D_\mathfrak{Z}d)\cdot  D_\mathfrak{Z} d, D_\mathfrak{Z} d\rangle_g= | D_\mathfrak{Z} d|_g^4=4$.
 Summarizing,  we have proved the following lemma.
 \begin{lemma}\label{dalpha}
 $$\begin{array}{rcl}
 \langle D_\mathfrak{Z}^2 d^{\alpha} \cdot D_\mathfrak{Z}d^{\alpha}, D_\mathfrak{Z}d^{\alpha}\rangle_g&=& 4 \alpha^3 (\alpha-1) d^{3\alpha -4} ,\\
  \langle (D_\mathfrak{Z}^2 d^{\alpha} )^*\cdot D_\mathfrak{Z}d^{\alpha}, D_\mathfrak{Z}d^{\alpha}\rangle_g&=& 4 \alpha^3 (\alpha-1) d^{3\alpha -4} .\end{array}$$
  \end{lemma}
  Choosing $\alpha=4/3$ we obtain 
  \begin{lemma}\label{fourthirds}
  \begin{equation*}\label{43} \Delta_{\mathfrak{Z}, \infty} d^{\frac43}= (\frac{4}{3})^4.
  \end{equation*} 
  \end{lemma}
 The following identity follows easily from the fact that $D_\mathfrak{X}^x d $ and $D_\mathfrak{X}^y d$ are unit vectors
 \begin{equation}\label{tensorsquared}
( D_\mathfrak{Z} d\otimes D_\mathfrak{Z} d)^2=  2 (D_\mathfrak{Z} d\otimes D_\mathfrak{Z} d ). \end{equation} 
 \begin{lemma}\label{dalphasquared}
 $$ 
 \langle (D_\mathfrak{Z}^2 d^{\alpha} )^2\cdot D_\mathfrak{Z}d^{\alpha}, D_\mathfrak{Z}d^{\alpha}\rangle_g= 8 \alpha^4 (\alpha-1)^2 d^{4\alpha -6}. $$
  \end{lemma} 
 \begin{proof} Let us compute $ (D_\mathfrak{Z}^2 d^{\alpha} )^2$:
 $$ 
 \begin{array}{rcl}
 (D_\mathfrak{Z}^2 d^{\alpha} )^2& =&  ( \alpha d^{\alpha-1} D_\mathfrak{Z}^2 d+ \alpha (\alpha -1) d^{\alpha-2} (D_\mathfrak{Z} d
 \otimes D_\mathfrak{Z}d) )^2\\
 & =&  \alpha^2 d^{2\alpha -2} (D^2_\mathfrak{Z} d)^2 + \alpha^2 (\alpha -1) d^{2\alpha -3}  D^2_\mathfrak{Z} d\, (D_\mathfrak{Z} d
 \otimes D_\mathfrak{Z}d) \\
 &&+ \alpha^2 (\alpha -1) d^{2\alpha -3} (D_\mathfrak{Z} d
 \otimes D_\mathfrak{Z}d)\,  D^2_\mathfrak{Z} d 
 +\alpha^2(\alpha-1)^2 d^{2\alpha-4}D_\mathfrak{Z} d
 \otimes D_\mathfrak{Z}d) )^2.
 \end{array}
 $$
 In the expression $ \langle (D_\mathfrak{Z}^2 d^{\alpha} )^2\cdot D_\mathfrak{Z}d^{\alpha}, D_\mathfrak{Z}d^{\alpha}\rangle$ there are four terms.
The first and third terms vanish because of  \eqref{eikonalbest}. The second term also vanishes since  $D^2_\mathfrak{Z} d\, (D_\mathfrak{Z} d
 \otimes D_\mathfrak{Z}d)=0$ by \eqref{matrixeikonal}. We are left with the fourth term
$$ \begin{array}{rcl}
 \alpha^2(\alpha-1)^2 \langle (D_\mathfrak{Z} d
 \otimes D_\mathfrak{Z}d) )^2\cdot D_\mathfrak{Z}d^{\alpha} ,D_\mathfrak{Z}d^{\alpha}\rangle_g& = & 2 \alpha^2(\alpha-1)^2d^{2\alpha-4} \langle (D_\mathfrak{Z} d
 \otimes D_\mathfrak{Z}d) \cdot D_\mathfrak{Z}d^{\alpha} ,D_\mathfrak{Z}d^{\alpha}\rangle_g\\
 &=& 2 \alpha^4(\alpha-1)^2 d^{4\alpha-6}  \langle (D_\mathfrak{Z} d
 \otimes D_\mathfrak{Z}d) \cdot D_\mathfrak{Z}d ,D_\mathfrak{Z}d\rangle_g \\
 &=& 8 \alpha^4(\alpha-1)^2 d^{4\alpha-6}  . \end{array}
 $$\end{proof}
  We record the identity we get taking $\alpha= 3/2$, although we will not need it in the rest of the paper,
\begin{equation}\label{threehalves}
 \langle (D_\mathfrak{Z}^2 d^{3/2} )^2\cdot D_\mathfrak{Z}d^{3/2}, D_\mathfrak{Z}d^{3/2}\rangle_g= \frac{81}{8}. 
 \end{equation}
 We will  also need to control a similar  term with the symmetrized second derivatives. We first consider    $\langle ((D_\mathfrak{Z}^2 d)^*)^2\cdot D_\mathfrak{Z}d, D_\mathfrak{Z}d\rangle_g$.
 \begin{lemma}\label{symmetrizedbound}
 Given a compact set $K\subset\mathbb{R}^n$ we can find a constant $C(K,\mathfrak{X})$ depending on $K$ and the frame $\mathfrak{X}$ so that $$0\le \langle ((D_\mathfrak{Z}^2 d)^*)^2\cdot D_\mathfrak{Z}d, D_\mathfrak{Z}d\rangle_g\le C(K,\mathfrak{X}).$$
  \end{lemma}
 \begin{proof} 
 The proof only uses basic properties of commutators of vector fields. Let us compute
 $$\begin{array}{rcl}
  \langle ((D_\mathfrak{Z}^2 d)^*)^2\cdot D_\mathfrak{Z}d, D_\mathfrak{Z}d\rangle_g &=&
   \langle (D_\mathfrak{Z}^2 d)^*\cdot D_\mathfrak{Z}d,     (D_\mathfrak{Z}^2 d)^*\cdot D_\mathfrak{Z}d\rangle_g\\
   &=& \displaystyle\sum_{i=1}^{2n} ((D_\mathfrak{Z}^2 d)^*\cdot D_\mathfrak{Z}d )_i^2\\
   &=& \displaystyle\sum_{i=1}^{2n} \left(\sum_{k=1}^{2n} ((D_\mathfrak{Z}^2 d)^*)_{ik} (D_\mathfrak{Z}d)_k\right)^2\\
   &=& \displaystyle\sum_{i=1}^{2n} \left(\sum_{k=1}^{2n} \left(\frac{X_i X_k d+ X_k X_i d}{2}        \right) X_k d\right)^2\\
    &=&  \displaystyle\sum_{i=1}^{2n}\left( \sum_{k=1}^{2n} \left(X_i X_k d-\frac{[X_i, X_k]d}{2}      \right) X_k d \right)^2  \\
      &=&  \displaystyle\sum_{i=1}^{2n}\left( \sum_{k=1}^{2n} X_i X_k d\,  X_k d - \frac{[X_i, X_k]d}{2} X_k d  \right)^2  \\  
       &=&  \displaystyle\sum_{i=1}^{2n}\left(\sum_{k=1}^n \frac{[X_i, X_k]d}{2} X_k d  \right)^2  \\          
       &\le &    C(K, \mathfrak{X}  ) ,     \end{array}
$$ 
 where we have used equation \eqref{eikonalbest} in the penultimate line  and the fact that $d$ is Lipschitz in the last line.
  
 \end{proof}
We have $(D^2_\mathfrak{Z} d^\alpha)^*= \alpha \, d^{\alpha-1} (D_\mathfrak{Z}^2 d)^*+ \alpha (\alpha -1)\, d^{\alpha-2} (D_\mathfrak{Z} d
 \otimes D_\mathfrak{Z}d)$  so that
 $$\begin{array}{rcl}
 ((D^2_\mathfrak{Z} d^\alpha)^*)^2 & = & \alpha^2  d^{2\alpha -2} ((D^2_\mathfrak{Z} d)^* )^2\\
 && + \alpha^2 (\alpha-1) d^{2\alpha -3} (D^2_\mathfrak{Z} d)^*(D_\mathfrak{Z} d \otimes D_\mathfrak{Z}d)\\
 && + \alpha^2 (\alpha-1) d^{2\alpha -3} (D_\mathfrak{Z} d \otimes D_\mathfrak{Z}d)(D^2_\mathfrak{Z} d)^* \\
 && + \alpha^2 (\alpha-1)^2 d^{2\alpha-4} (D_\mathfrak{Z} d \otimes D_\mathfrak{Z}d)^2. \end{array}
 $$ 
 Next, we observe that by \eqref{eikonalbest} we have
 $$\langle (D^2_\mathfrak{Z} d)^*(D_\mathfrak{Z} d \otimes D_\mathfrak{Z}d) D_\mathfrak{Z} d, D_\mathfrak{Z} d \rangle_g=0$$ and
 by \eqref{matrixeikonal} and \eqref{eikonalbest} we also have
  $$\langle (D_\mathfrak{Z} d \otimes D_\mathfrak{Z}d)(D^2_\mathfrak{Z} d)^* D_\mathfrak{Z} d, D_\mathfrak{Z} d \rangle_g=0.$$ 
  Using \eqref{tensorsquared} we conclude that
  $$
  \begin{array}{rcl}
  \langle ((D^2_\mathfrak{Z} d^\alpha)^*)^2 D_\mathfrak{Z} d  ,D_\mathfrak{Z} d  \rangle_g & = & \alpha^2  d^{2\alpha -2}  \langle ((D^2_\mathfrak{Z} d)^*)^2 D_\mathfrak{Z} d  ,D_\mathfrak{Z} d  \rangle_g + 2  \alpha^2 (\alpha-1)^2
   \langle (D_\mathfrak{Z} d \otimes D_\mathfrak{Z}d) D_\mathfrak{Z} d ,D_\mathfrak{Z}  d \rangle_g \\
  &=& \alpha^2  d^{2\alpha -2}  \langle ((D^2_\mathfrak{Z} d)^*)^2 D_\mathfrak{Z} d  ,D_\mathfrak{Z} d  \rangle_g + 8  \alpha^2 (\alpha-1)^2 d^{2\alpha-4}.  \\
  \end{array} 
   $$
   Hence, we can conclude this section with the following result whose proof immediately follows from the previous equality.
  \begin{lemma}\label{symmetrizedboundalpha}
 Given a compact set $K\subset\mathbb{R}^n$ we can find a constant $c_0=C(K,\mathfrak{X})$ depending on $K$ and the frame $\mathfrak{X}$ so that $$0\le \langle ((D_\mathfrak{Z}^2 d^\alpha)^*)^2\cdot D_\mathfrak{Z}d, D_\mathfrak{Z}d\rangle_g\le c_0\,\alpha^2  d^{2\alpha -2}+8  \alpha^2 (\alpha-1)^2 d^{2\alpha-4}$$ and
 $$0\le \langle ((D_\mathfrak{Z}^2 d^\alpha)^*)^2\cdot D_\mathfrak{Z}d^\alpha, D_\mathfrak{Z}d^\alpha \rangle_g\le c_0\,\alpha^4  d^{4\alpha -4}+8  \alpha^4 (\alpha-1)^2 d^{2\alpha-6}.$$  \end{lemma}

  \section{Viscosity Solutions and Frames  }
  We are studying viscosity solutions of the equation
  \begin{equation}\label{mainequation}
   \Delta_{\mathfrak{X},\infty} u(x) =f(x,u(x), D_{\mathfrak{X}}u(x))
      \end{equation}
   where $f$ is a continuous function satisfying the growth condition \eqref{gradproblem}. 
   We assume that $u$ is a viscosity solution as in Definition \ref{viscositytestfunctions0}. 
       \par 
   We can  use jets adapted to the frame $\mathfrak{X}$ to characterize viscosity sub and supersolutions.   
To define second order superjets of an upper-semicontinuous function $u$,
consider smooth functions $\varphi$ touching $u$ from above at a point $x_0$. The second-order super-jet of the upper-semicontinuous function $u$ at the point $x_0$ is the set 
\begin{align*}
K_{\mathfrak{X}}^{2,+}(u,x_0)  = \bigg\lbrace (D_{\mathfrak{X}} \varphi(x_0), (D_\mathfrak{X}^2\varphi(x_0))^*)
\colon
\varphi \in C^2& \text{\ in a neighborhood of} \ x_0,\  \varphi(x_0) = u(x_0),\\
\varphi(x) & \geq u(x)   \text{  in a neighborhood of} \
x_0\bigg\rbrace.
\end{align*}

For each function $\varphi\in C^2$  and a point $x_0$ we write 
\begin{equation}\label{fitojet}\begin{array}{rcccl}
\eta&  = &   D_{\mathfrak{X}} \varphi(x_0) & = & \big(X_1\varphi(x_0),X_2\varphi(x_0),\ldots, X_n\varphi(x_0)\big)\\
 A_{ij}&= & (D_\mathfrak{X}^2\varphi(x_0))^* &  = & \frac{1}{2} \big(X_i(X_j(\varphi))(x_0)+X_j(X_i(\varphi))(x_0)\big).
\end{array}
\end{equation}
 This representation clearly depends on the frame $\mathfrak{X}$. Using
the Taylor theorem (Lemma \ref{taylor})  for $\varphi$ and the fact that $\varphi$ touches $u$ from
above at $x_0$ we get
\begin{equation}\label{jjet}
u\left(\Theta_{x_0}(t)\right) \le 
u(x_0)+\langle \eta  ,t  \rangle+
\frac{1}{2}\langle  X  t , t  \rangle
+o(|t|^2), \text{ as }t\to 0.
\end{equation}
We may also consider $J_{\mathfrak{X}}^{2,+}(u,x_0)$ defined as the collections
of pairs $(\eta,X)$ such that \eqref{jjet} holds.
 Denoting  by $J^{2,+}(v,  t)$ the standard Euclidean superjets we also get from \eqref{jjet} the equivalence 
 \begin{equation}\label{ofjets} 
(\eta, X)\in J_{\mathfrak{X}}^{2,+}(u,x_0) \iff (\eta, X) \in J^{2,+}(u\circ \Theta_{x_0}, 0)
\end{equation}
Using the identification given by
\eqref{fitojet} it is clear that
$$K_{\mathfrak{X}}^{2,+}(u,x_0) \subset J_{\mathfrak{X}}^{2,+}(u,x_0).$$ In fact, we have equality.
This is the analogue of the Crandall-Ishii Lemma of \cite{Cr97} that was extended to vector fields in \cite{BBM05}:
\begin{lemma}\label{crandall}
$$K_{\mathfrak{X}}^{2,+}(u,x_0)= J_{\mathfrak{X}}^{2,+}(u,x_0).$$
\end{lemma}
 We define second order subjets  $J_{\mathfrak{X}}^{2,-}(u,x_0)$ similarly. We are in position to introduce the equivalent definition of viscosity solution  based on jets to our $\infty$-Laplace equation.\par
 
  \begin{definition}\label{viscositytestjets}
   An upper semi-continuous function $u$ is a viscosity subsolution of \eqref{mainequation} in a domain $\Omega\subset\mathbb{R}^n$  if whenever $(\eta, X)\in J_{\mathfrak{X}}^{2,+}(u,x_0)$  for  $x_0\in \Omega$ we have 
   $$  \langle X\cdot \eta, \eta\rangle_g \ge f(x_0, u(x_0), \eta).$$
 A lower semi-continuous function $v$ is a viscosity supersolution of \eqref{mainequation} in a domain $\Omega$  whenever $(\eta, X)\in J_{\mathfrak{X}}^{2,-}(v,x_0)$  for  $x_0\in \Omega$ we have    $$  \langle X\cdot \eta, \eta\rangle_g \le f(x_0, v(x_0), \eta).$$
   \end{definition}
 \par
 
 We shall need the Euclidean Theorem of Sums (see \cite{CIL92}) that we state for functions defined on $D=B_1(0)$ the Euclidean ball of radius $1$ centered at the origin. 
\begin{theorem}\label{sums}
 Let $u$ be upper-semicontinuous  and $v$ be lower-semicontinuous functions in $B_1$. Let $\phi\in C^2(\mathbb{R}^n \times \mathbb{R}^n)$ and suppose that there is a point $(\hat{x}, \hat{y})\in B_1 \times B_1$ such that 
$$u(\hat{x})-v(\hat{y})-\phi(\hat{x}, \hat{y})= \max_{(x,y)\in \overline{B}_1\times \overline{B}_1} \left(
u(x)-v(y)-\phi(x,y)\right). $$
Then for each $\mu>0$ there are symmetric matrices $X_\mu$ and $Y_\mu$ such that
$$(D_x \phi(\hat x, \hat y), X_\mu)\in  \overline{J}^{2, +}(u, \hat{x}), \hskip .3in(-D_y \phi(\hat x, \hat y), Y_\mu)\in  \overline{J}^{2, -}(v, \hat{y}), $$ and we have the estimate

$$
-(\mu + \| D^2\varphi(\hat{x}, \hat{y})\|) \left(\begin{array}{cc}
I & 0      \\
0 & -I
\end{array}\right)
\le 
\left(
\begin{array}{cc}
X_\mu & 0      \\
0 & -Y_\mu
\end{array}
\right)
\le D^2\phi(\hat{x}, \hat{y}) +\frac{1}{\mu} (D^2\phi(\hat{x}, \hat{y}))^2.
$$

\end{theorem}
  
 \section{Lipschitz Estimate: Proof of Theorem \ref{apriori0}} Let $u$ be a viscosity solution of the equation
 \begin{equation}\label{unouno}
  \Delta_{\mathfrak{X},\infty} u(x)=f (x )
 \end{equation}
  in a domain $\Omega\subset\mathbb{R}^n$.  We shall assume that
 $B_{2 \,i(x_0)}\subset\Omega$.  The strategy of the proof taken from \cite{L14} is to reduce the problem to the case when $f\ge 0$, so that
 $u$ is a viscosity subsolution. It the follows from a comparison with the distance function that $u$ is Lipschitz.
 \par
 We add a new variable $x_{n+1}$ and a new vector field $X_{n+1}=\frac{\partial}{\partial x_{n+1}}$. Consider the function
 $$v(x_1, \ldots, x_n, x_{n+1})= u(x_1, \ldots, x_n) +  c \, |x_{n+1}|^{4/3},$$ where $c$ is constant and the extended frame
 $\mathcal{Y}=\{X_1, \ldots, X_n, X_{n+1}\}$.  This frames induces a Riemannian metric $h$ that satisfies
 $$\langle (\xi_1,\eta_1), (\xi_2,\eta_2)\rangle_h=\langle (\xi_1,\xi_2)\rangle_g + \eta_1 \eta_2$$ for $\xi_1, \xi_2\in \Rn$ and
 $\eta_1, \eta_2\in\R$.  In the smooth case we have the identity
 \begin{equation}\label{extension}
 \begin{array}{rcl}
 \langle D^2_\mathcal{Y} v \, D_\mathcal{Y}v, D_\mathcal{Y}v\rangle_h
 & = & \langle D^2_\mathfrak{X} u \, D_\mathfrak{X}u, D_\mathfrak{X}u\rangle_g+ (X^{n+1}X^{n+1}v) (X^{n+1} v)( X^{n+1}v)\\
 & & \\
 & = &  f+c^3 (\frac{4}{3})^3
 
 \end{array}
 \end{equation}
 In fact, this is also true in the viscosity sense. If a function $u$ is a viscosity solution of \eqref{unouno},  the extended function $v$ is a viscosity solution of \eqref{extension}, 
 see Chapter 10  in \cite{L16}.  
 Thus,
 we can assume that $f\ge 0$ by taking an appropriate constant $c$ depending only on $\|f \|_{L^\infty(B_{i(x_0)}(x_0))}$.\par
 Therefore, we may assume  that $u$ is a subsolution of the $\infty$-Laplacian relative to the frame $\mathcal{X}$. 
 Consider the functions    $w(y)= u(x)- u(x_0)$ and $z(y)= M_r \frac{d(x_0, y)}{r}$ on the ball $B_{r}(x_0)$ for $r<i(x_0)/2$, where
  $$M_r=\sup\{w(x)\colon d(x_0,x)=r\}.$$
 We compare these functions in the puncture ball
 $B_{r}(x_0)\setminus\{x_0\}$, where $w$ is $\infty$-subharmonic and $u$ is $\infty$-harmonic. We see that $w \le z$ on $\partial  B_{r}(x_0)\setminus\{x_0\}$ and thus in $B_{r}(x_0)$ by the comparison principle.
 We conclude that
 $$ u(x)-u(x_0) \le M_r \frac{d(x_0, x) }{r}$$  for all $x\in B_{r}(x_0)$. The constant $M_r$ depends only on the $L^\infty$-norm of $u$ on $\overline{B_{i(x_0)/2}(x_0)}$. 
  Using a similar argument for $-u$ we get
 $$\frac{|u(x)-u(x_0)|}{d(x,x_0)}\le \frac{M_{i(x_0)/4}}{i(x_0)/4}\le \frac{ 4 \, \| u \|_{L^\infty(\overline{B_{i(x_0)/2}(x_0)})}}   {i(x_0)}$$
 for all $x\in B_{i(x_0)/4}$. We deduce  the following bound of the local Lipschitz constant at $x_0$
 
 $$\text{Lip } u(x_0)= \lim_{r\to 0^+} \sup_{y\in B_r\setminus\{x_0\}} \frac{ |u(x)-u(x_0)|} {d(x,x_0)}\le
 \frac{ 4 \, \| u \|_{L^\infty(B_{i(x_0)}(x_0))}}   {i(x_0)}.$$
 
By compactness we have $\kappa(x_0)=\inf\{ i(y)\colon y\in B_{i(x_0)}(x_0)\} >0$. Thus, for all $y\in  B_{i(x_0)}(x_0)$ we obtain
 $$\text{Lip } u(y) \le   \frac{ 4 \, \| u \|_{L^\infty(B_{2\, i(0)}(x_0))}}   {\kappa(x_0)}.$$ Therefore, we obtain
 $$\text{ess}\!\!\!\!\!\!\!\!\sup_{y\in B_{i(x_0)}(x_0)}  \text{Lip } u(y)\le  \frac{ 4 \, \| u \|_{L^\infty(B_{2\, i(0)}(x_0))}}   {\kappa(x_0)}.$$
 From Theorem 4.7 in \cite{DMV13} we deduce that
 $$|D_\mathfrak{X}u(y) |_g\le  \frac{ 4 \, \| u \|_{L^\infty(B_{2\, i(0)}(x_0))}}   {\kappa(x_0)}$$ for a.e. $y$, from which it follows that
we can take
$$L=  \frac{ 4 \, \| u \|_{L^\infty(B_{2\, i(0)}(x_0))}}   {\kappa(x_0)}.$$
\normalcolor
 \section{H\"older Estimate: Proof of Theorem \ref{apriori1}} 
 For $\alpha\in (0,1)$, positive constants $L$ and $A$ to be determined later, and $z\in B_{i(x_0)/4}$ consider the penalization function
$$G(x,y)=L \, d^\alpha (x,y)+ A \, d^2(x,z).$$
Suppose that
\begin{equation}\label{claim}
u(\hat{x})- u( \hat{y})- G(\hat{x}, \hat{y}) =\sup\left\{u(x)-u(y)-G(x,y)\colon (x,y)\in \overline{B_{i(x_0)}}\times \overline{B_{i(x_0)}}
\right\}=\theta>0.
\end{equation}
We will show that \eqref{claim} leads to a contradiction for specific choices of $L$ and $A$ depending only on $\|u \|_{L^\infty(B_{i(x_0)})}$, 
 $\| f \|_{L^{]\infty}(B_{i(x_0)})}$, and $C(\overline{B_{2\,i(x_0)}},g)$ when $\alpha\in(0,1)$. 
When \eqref{claim} does not hold we have
$$u(x)-u(y)\le L \, d^\alpha (x,y)+ A \, d^2(x,z), \text{  for  } x,y\in B_{i(x_0)}.$$ Letting $x=z$ we get the theorem. \par

Let us now  assume that \eqref{claim} holds. Since $G(x,y)\ge 0$ we must have $\hat{x}\not=\hat{y}$. 
In what follows we temporarily omit the center $x_0$ of the balls under consideration. 
\begin{claim} For $A\ge 8 \frac{\|u \|_{L^\infty(B_{i(x_0)})}}{i(x_0)^2}$ we have $\hat{x}\in B_{(3/4)i(x_0)}$.

\end{claim} Suppose  $\hat{x}\notin B_{(3/4)i(x_0)}$, then $d(\hat{x}, z)\ge (1/2)i(x_0)$ so that we get
$$0<\theta= u(\hat{x})-u(\hat{y})- L \, d^\alpha (\hat{x},\hat{y})- A \, d^2(\hat{x},z), $$ and
$$0\le 2 \|u \|_{L^\infty(B_1)}-L \, d^\alpha (\hat{x},\hat{y})-\frac{A}{4}.$$  This implies
$A <  8 \frac{\|u \|_{L^\infty(B_{i(x_0)})}}{i(x_0)^2}$. \par
\textbf{From now on we take $A= 8 \frac{\|u \|_{L^\infty(B_{i(x_0)})}}{i(x_0)^2}$. }
\begin{claim}
For $L \ge 16 \|u \|_{L^\infty(B_{i(x_0)})}$ we have $\hat{y}\in B_{(7/8)i(x_0)}$.
\end{claim}
 If $\hat{y}\notin B_{(7/8)i(x_0)}$, we have $d(\hat{y}, x_0)\ge (7/8)i(x_0)$ so that $d(\hat{x}, \hat{y})\ge (1/8)i(x_0)$.  From  the inequality
 $$0<\theta= u(\hat{x})-u(\hat{y})- L \, d^\alpha (\hat{x},\hat{y})- A \, d^2(\hat{x},z) $$  we obtain 
 $$0< 2 \|u \|_{L^\infty(B_{i(x_0)})}- L \, d^\alpha (\hat{x},\hat{y}). $$  This implies
 $$L < 
 \frac{2 \|u \|_{L^\infty(B_{i(x_0)})}}{d^{\alpha} (\hat{x},\hat{y})} 
  < \frac{2 \|u \|_{L^\infty(B_{i(x_0)})}}{((1/8)i(x_0))^{\alpha}}
  = 2\, 8^\alpha 
  \frac{ \|u \|_{L^\infty(B_{i(x_0)})}}{i(x_0)^\alpha}
  < 16   \frac{ \|u \|_{L^\infty(B_{i(x_0)})}}{i(x_0)^\alpha}.$$
  
  \par
 \textbf{From now on we take $L\ge L_0= 16 \frac{ \|u \|_{L^\infty(B_{i(x_0)})}}{i(x_0)^\alpha}$. }\par
Therefore we can assume  that $u(x)-u(y)-G(x,y)$ has an interior positive maximum at the point $(\hat{x}, \hat{y})$ for our choices of $A$ and $L\ge L_0$.
Note that we always have 
$$ L \, d^\alpha (\hat{x},\hat{y})+ A \, d^2(\hat{x},z) \le 2 \, \|u \|_{L^\infty(B_{i(x_0)})},$$
and that the point $(\hat{x}, \hat{y})$ where the maximum is achieved depends on $L,A, \alpha, z$ and $u$. The function $u$ and the values of  $A$, $\alpha$ and $z$ will remain fixed in the our arguments below. We will eventually let $L\to\infty$. From now on we will denote the point of maximum
$$(x_L, y_L),$$
where of course the subindex $L$  denotes the dependence on $L$. In particular, we have
\begin{equation}\label{zeroorder}
L \, d^\alpha (x_L,y_L)  \le 2 \, \|u \|_{L^\infty(B_{i(x_0)})},
\end{equation} so that we have
\begin{equation}\label{referee}
\lim_{L\to\infty} d(x_L,y_L) =0.\end{equation}
By selecting a sequence $L_m\to\infty$ we conclude the existence of a point $x^*\in B_{(3/4)i(x_0)}$ such that 
$$x^*=\lim_{m\to\infty} x_{L_m}= \lim_{m\to\infty}y_{L_m}. $$ We will omit the subindex $m$ and write just 
$L$ for $L_m$. Note that we may assume that $x_L$ and $y_L$ are in the ball $B_{i(x^*)/4}(x^*)$ for $L$ large enough. \par

Consider next the flow exponentials $s\mapsto\Theta_{x_L}(s)$ and $t\mapsto\Theta_{y_L}(t)$ defined in a neighborhood of zero. The function
$u(x)-u(y)-G(x,y)$ has a positive local maximum at $(x_L, y_L)$ if and only if the function
$$u(\Theta_{x_L}(s))-u(\Theta_{y_L}(t))-G(\Theta_{x_L}(s), \Theta_{y_L}(t))$$ has a positive local maximum at $(0,0)$.
\par

From the equivalence \eqref{ofjets} we note the $\mathfrak{Z}$ second order sub and superjets of the function $L \,d^\alpha(x_L, y_L)$ at 
the point $(x_L, y_L)$ are the same as the Euclidean second order sub and superjets of the function
$$\phi(s,t)=  L \, d^\alpha (\Theta_{x_L}(s),\Theta_{y_L}(t))$$ so that
$G(\Theta_{x_L}(s), \Theta_{y_L}(t))= \phi(s,t)+ A \, d^2(\Theta_{x_L}(s),z)$.
\par
Next we write $w(x)= u(x)- A\, d^2(x,z)$ so that 
\begin{equation*}
\begin{array}{rcl}
     u(x)-u(y)-G(x,y)  &=& u(x)-A\, d^2(x,z) -u(y) - L\, d^\alpha(x,y)   \\
& =& w(x)-u(y)-d^\alpha(x,y).
\end{array}
\end{equation*}
We are now ready to apply the Theorem of Sums \ref{sums} to the difference
$$w(\Theta_{x_L}(s))-u(\Theta_{y_L}(t))-\phi(s,t)$$ at the point $(0,0)$.
For each $\mu>0$, there exists symmetric $n \times n$ matrices $X_\mu$ and $Y_\mu$ so that
$$(D_s\phi(0,0), X_\mu)\in  \overline{J}^{2, +}(w\circ \Theta_{x_L}, 0), \hskip .3in(-D_t \phi(0,0), Y_\mu)\in  \overline{J}^{2, -}(u\circ \Theta_{y_L}, 0), $$
 and we have the estimate
\begin{equation}\label{sums2}
-(\mu + \| D^2\phi(0, 0)\|) \left(\begin{array}{cc}
I & 0      \\
0 & -I
\end{array}\right)\le
\left(
\begin{array}{cc}
X_\mu & 0      \\
0 & -Y_\mu
\end{array}
\right)
\le D^2\phi(0,0) +\frac{1}{\mu} (D^2\phi(0, 0))^2.
\end{equation}
Using the equivalence \eqref{ofjets} we translate back to the frame sub and super jets and set:
\begin{equation}\label{transfer}
\begin{array}{rcccl}
\xi_L&=&D_s\phi(0,0)& = & L\, D_\mathfrak{X}^x d^\alpha(x_L, y_L),\\
\eta_L &=& D_t\phi(0,0)& = & L\, D_\mathfrak{X}^y d^\alpha(x_L, y_L),\\
&&(\xi_l, X_\mu)& \in &   \overline{J}^{2, +}(w, x_L),\\
&&(\eta_L, Y_\mu)& \in &   \overline{J}^{2, -}(u, y_L).
\end{array}
\end{equation}
The  second order Taylor expansion of $\phi(s,t)$ at the point $(0,0)$ using the equivalence \eqref{ofjets} can be written as
$$\phi(s,t)= \langle (\xi_L,\eta_L),(s,t)\rangle+
\frac{1}{2}\langle   L\,  D^{2}_{\mathfrak{Z}}d^\alpha(x_L, y_L)\cdot (s,t)      ,  (s,t)    \rangle+ o(|s|^2+|t|^2)\text{ as } s,t\to 0 ,$$ from which it follows that
\begin{equation}\label{secondderivative}
D^2\phi(0,0)= L\,  (D^{2}_{\mathfrak{Z}}d^\alpha(x_L, y_L))^*=M_L.
\end{equation}
Note that the matrix $D^{2}_{\mathfrak{Z}}d^\alpha(x_L, y_L)$ is  not symmetric in general, so we must symmetrize it.
We can rewrite the third line  in \eqref{transfer} as
\begin{equation}\label{transfer2}
 (\xi_L+ A\, D^x_\mathfrak{X} d^2(x_L,z), \,\,X_\mu+ A\, (D^{2,x}_{\mathfrak{X}}d^2(x_L, z))^*\in   \overline{J}^{2, +}(u, x_L)
 \end{equation}
and rewriting the inequalities \eqref{sums2} as
\begin{equation}\label{sums3}
-(\mu + \|M_L\| )\left(\begin{array}{cc}
I & 0      \\
0 & -I
\end{array}\right)  \le 
\left(
\begin{array}{cc}
X_\mu & 0      \\
0 & -Y_\mu
\end{array}
\right)
\le M_L +\frac{1}{\mu} (M_L^2).
\end{equation} 

Using the fact that $u$ is viscosity subsolution of \eqref{mainequation} and \eqref{transfer2} to get
\begin{equation}\label{subcondition}
\begin{array}{rl}
f(x_L,& \!\!\!\!\!u(x_L), \xi_L + A\, D^x_\mathfrak{X} d^2(x_L,z))    \\ & \\
\le &  \langle \left[X_\mu+ A\, (D^{2,x}_{\mathfrak{X}}d^2(x_L, z))^*\right]\cdot 
\left(\xi_L+ A\, D^x_\mathfrak{X} d^2(x_L,z)\right), 
 \left(\xi_L+ A\, D^x_\mathfrak{X} d^2(x_L,z)\right)\rangle_g.
 \end{array} 
\end{equation}
Using  the fact that $u$ is viscosity supersolution of \eqref{mainequation} and the fourth statement in \eqref{transfer} we obtain
\begin{equation}\label{supercondition}
f(y_L, u(y_L), \eta_L) \ge  \langle Y_\mu\cdot 
\eta_L ,  \eta_L\rangle_g.
\end{equation}
Adding these estimates we get
\begin{equation}\label{main10}
\begin{array}{rl}
f(x_L, &\!\!\!\!\!u(x_L), \xi_L + A\, D^x_\mathfrak{X} d^2(x_L,z))-f(y_L, u(y_L), \eta_L)  \\ & \\
&\le  \langle \left[X_\mu+ A\, (D^{2,x}_{\mathfrak{X}}d^2(x_L, z))^*\right]\cdot 
\left(\xi_L+ A\, D^x_\mathfrak{X} d^2(x_L,z)\right), 
 \left(\xi_L+ A\, D^x_\mathfrak{X} d^2(x_L,z)\right)\rangle_g \\   & \\
 & -  \langle Y_\mu\cdot 
\eta_L , L\, \eta_L\rangle.\end{array}
\end{equation}
Expanding the right hand side of \eqref{main10}  we obtain  
\begin{equation}\label{main11}
\begin{array}{rcl}
&  &\langle X_\mu\cdot \xi_L, \xi_L\rangle_g- \langle   Y_\mu\cdot \eta_L, \eta_L\rangle_g\\
&& +2\,A\, \langle X_\mu \cdot \xi_L, D^x_\mathfrak{X} d^2(x_L,z)\rangle_g\\
&& +A^2 \langle X_\mu\cdot D^x_\mathfrak{X} d^2(x_L,z),D^x_\mathfrak{X} d^2(x_L,z)\rangle_g\\
&& +A\, \langle  (D^{2,x}_{\mathfrak{X}}d^2(x_L, z))^*\cdot
 \xi_L, \xi_L\rangle_g \\
 && +2 A^2 \,\langle  (D^{2,x}_{\mathfrak{X}}d^2(x_L, z))^*\cdot
 \xi_L, D^x_\mathfrak{X} d^2(x_L,z)\rangle_g  \\
 &&+A^3 \langle  (D^{2,x}_{\mathfrak{X}}d^2(x_L, z))^*\cdot
 D^x_\mathfrak{X} d^2(x_L,z), D^x_\mathfrak{X} d^2(x_L,z)\rangle_g \\
 &=& T_1+T_2+T_3+T_4+T_5+T_6.
 \end{array}
\end{equation}

 \begin{claim}\label{t1} Estimate of $T_1$:
 $$\begin{array}{rcl}
 T_1& \le&   4\alpha^3(\alpha-1)\, L^3 d^{3\alpha-4}+ \frac{L^4}{\mu}\left(
 c_0\,\alpha^4  d^{4\alpha -4}+8  \alpha^4 (\alpha-1)^2 d^{2\alpha-6} \right)\\
 &=& 4(\alpha-1)\alpha^3 L^3 d^{3\alpha-4}\left(1+\frac{2\,L\, \alpha (\alpha-1) d^{\alpha-2}}{\mu}+
\frac{c_0\, \alpha\, L\, d^\alpha}{\mu \, 4(\alpha-1)}
 \right)
 \end{array}$$
 \end{claim}
\begin{proof}
From the upper bound in \eqref{sums3} we get
$$\langle X_\mu\cdot \xi_L, \xi_L\rangle_g- \langle   Y_\mu\cdot \eta_L, \eta_L\rangle_g  \le
\langle \left( M_L +\frac{1}{\mu} (M_L^2)\right) \cdot  \left(
\begin{array}{c}
\xi_L    \\
\eta_L
\end{array}
\right)   ,    \left(
\begin{array}{c}
\xi_L  \\
\eta_L
\end{array}
\right)    \rangle_g
$$
Recall that $
M_L = L(D^{2}_{\mathfrak{Z}}d^\alpha(x_L, y_L))^*$. We need to estimate
$$\langle M_L \cdot  \left(
\begin{array}{c}
\xi_L    \\
\eta_L
\end{array}
\right)   ,    \left(
\begin{array}{c}
\xi_L  \\
\eta_L
\end{array}
\right)    \rangle_g \text{      and       } \langle M_L^2 \cdot  \left(
\begin{array}{c}
\xi_L    \\
\eta_L
\end{array}
\right)   ,    \left(
\begin{array}{c}
\xi_L  \\
\eta_L
\end{array}
\right)    \rangle_g.
$$
Using Lemma \ref{dalpha}, we get 
$$
\begin{array}{rcl}
\langle M_L \cdot  \left(
\begin{array}{c}
\xi_L    \\
\eta_L
\end{array}
\right)   ,    \left(
\begin{array}{c}
\xi_L  \\
\eta_L
\end{array}
\right)    \rangle_g & = & L^3\langle  (D^{2}_{\mathfrak{Z}}d^\alpha(x_L, y_L))^* \cdot  \left(
\begin{array}{c}
 D_\mathfrak{X}^x d^\alpha(x_L, y_L)  \\
 D_\mathfrak{X}^y d^\alpha(x_L, y_L)
\end{array}
\right)   ,    \left(
\begin{array}{c}
D_\mathfrak{X}^x d^\alpha(x_L, y_L) \\
 D_\mathfrak{X}^y d^\alpha(x_L, y_L)
\end{array}
\right)    \rangle_g\\
&=& L^3\, 4\alpha^3(\alpha-1) d(x_L,y_L)^{3\alpha-4}, 
\end{array}
$$
and invoking Lemma \ref{symmetrizedboundalpha} we get
$$
\langle M_L^2 \cdot  \left(
\begin{array}{c}
\xi_L    \\
\eta_L
\end{array}
\right)   ,    \left(
\begin{array}{c}
\xi_L  \\
\eta_L
\end{array}
\right)    \rangle_g  \le  L^4\left( c_0\,\alpha^4  d^{4\alpha -4}+8  \alpha^4 (\alpha-1)^2 d^{2\alpha-6}\right).
$$

\end{proof}
For a fixed $\beta\in\mathbb{R}$ to be determined below set  $\mu=\beta  \left(2\, L \, \alpha (\alpha-1) d^{\alpha-2}\right)$,  so that we have

$$
\begin{array}{rcl}
1+\frac{2\alpha (\alpha-1) \, L \,d^{\alpha-2}}{\mu}+
 \frac{c_0\, \alpha\,L\, d^\alpha}{\mu \, 4(\alpha-1)}& =&1+\frac{1}{\beta}+
  \frac{c_0\, \alpha\,  d^\alpha}{\beta \left(2\alpha (\alpha-1) d^{\alpha-2}\right) \, 4(\alpha-1)}\\
 &=& 1+\frac{1}{\beta}+
 \frac{c_0\, d^2}{8\, \beta (\alpha-1)^2 } \end{array}.
 $$

Since  $\beta<0 $ we have
$$1+\frac{1}{\beta}+\frac{c_0\, d^2}{8\, \beta (\alpha-1)^2 }\ge 1+\frac{1}{\beta}+
  \frac{c_0}{8\, \beta (\alpha-1)^2 }.
$$
We can now choose $\beta$ depending only on $c_0$ and $\alpha$ so that 
\begin{equation}\label{newkey}
1+\frac{2 \,L\,\alpha (\alpha-1) d^{\alpha-2}}{\mu}+
\frac{L\,c_0\, \alpha\, d^2}{\mu \, 4(\alpha-1)}\ge \frac{1}{2}.
\end{equation}
Our next task is to estimate the norm $\|X_\mu\|$ using Proposition \ref{boundford2} and \eqref{sums3}.
Using the pair of vectors $(v,0)\in \mathbb{R}^n\times\mathbb{R}^n$ in \eqref{sums3} we get
$$ -(\mu+ \|M_L\|) |v|_g^2\le \langle X_\mu\cdot v, v\rangle_g \le \left(\|M_L\|+\frac{\|M_L\|^2}{\mu}\right) |v|_g^2.$$
We estimate the norm of $M_L$ by using Proposition \ref{boundford2}
$$\begin{array}{rcl}
 \| M_L\| & \le  & L \|D^{2}_{\mathfrak{Z}}d^\alpha(x_L, y_L))^*\|\\
 &\le & L\|  \alpha \,d^{\alpha-1} (D_\mathfrak{Z}^2 d)^*+ \alpha (\alpha -1) d^{\alpha-2} (D_\mathfrak{Z} d
 \otimes D_\mathfrak{Z}d) \|\\
 &\le & L\,  \alpha \, d^{\alpha-1}\| (D_\mathfrak{Z}^2 d)^* \|+L\, \sqrt{2}\,\alpha\, |\alpha-1| d^{\alpha -2}\\
 &\le & L\,  \alpha \, d^{\alpha-1}\frac{c_1}{d}+ L\,\sqrt{2}\,\alpha \,|\alpha-1| d^{\alpha -2}\\
 & \le &  c_2 L\,  \alpha \, d^{\alpha-2} 
 \end{array}
 $$
for some constant $c_2\ge 1$. Note that we can choose $\beta$ sufficiently negative so that $c_2\le \beta (\alpha -1)$ we can guarantee that 
$\| M_L\|\le \mu/2$. 
For the upper bound  we compute
$$
\begin{array}{crl}
\langle X_\mu\cdot v, v\rangle_g  & \le  & \left( c_2 L\,  \alpha \, d^{\alpha-2} + \frac{(c_2 L\,  \alpha \, d^{\alpha-2} )^2}{\mu}\right) |v|_g^2\\
	&\le&  \left( \|M_L\|+\frac{\|M_L\|^2}{\mu} \right) |v|_g^2\\
	&\le & \frac34 \, \mu  |v|_g^2.
\end{array}$$
For the lower bound $$\begin{array}{rcl}
\langle X_\mu\cdot v, v\rangle_g  & \ge & -(\mu+ \|M_L\|) \, |v|_g^2\\
&\ge & -\frac{3}{2} \mu\,  |v|_g^2.
\end{array}
$$ Combining both estimates we get
\begin{equation}\label{normofxmu}
\| X_\mu\| \le \frac34 \, \mu\le \frac34 \, \beta  \left(2\, L \, \alpha (\alpha-1) d^{\alpha-2}\right)\le c_4 \, L \, \alpha \, d^{\alpha-2}.
\end{equation}
\begin{claim}\label{t2} Estimate of $T_2$:
 $$ 
 T_2 \le c_5 \, \alpha^2\, L^2\, d^{2\alpha -3}.
 $$
 \end{claim}
 \begin{claim}\label{t3} Estimate of $T_3$:
 $$ 
 T_3 \le c_6 \, \alpha^2\, L\, d^{\alpha -2}.
 $$
 \end{claim}
  \begin{claim}\label{t4} Estimate of $T_4$:
 $$ 
 T_4 \le c_7 \, \alpha^2\ L^2\,d^{2\alpha -2}.
 $$
 \end{claim}
   \begin{claim}\label{t5} Estimate of $T_5$:
 $$ 
 T_5 \le c_8 \, \alpha\, L \,d^{\alpha -1}.
 $$
 \end{claim}
    \begin{claim}\label{t6} Estimate of $T_6$:
 $$ 
 T_6 \le c_9.
 $$
 \end{claim}
 Let us now estimate the left-hand side of \eqref{main10} using condition \eqref{gradproblem}. We  have as $L\to\infty$
  \begin{equation*}\label{lhs1}
 \begin{array}{rcl}
 | f(x_L, u(x_L),  \xi_L + A\, D^x_\mathfrak{X} d^2(x_L,z))|  &\le &  C_2 |  \xi_L + A\, D^x_\mathfrak{X} d^2(x_L,z))|_g^\beta + C_3\\
 &=& C_2 |  L\, D_\mathfrak{X}^x (x_L,y_L)d^\alpha+ A\, D^x_\mathfrak{X} d^2(x_L,z))|_g^\beta + C_3\\
 &\le & C_4| L \alpha d^{\alpha-1} D_\mathfrak{X}^x d(x_L, y_L)|^\beta + C_5 |d(x_L,z) D^x_\mathfrak{X} d(x_L,z))|^\beta+C_3\\
 &\le & C_6 L^\beta d^{\beta(\alpha-1)} + C_7 \\
  &\le & C_8  L^\beta d^{\beta(\alpha-1)}, 
 \end{array} 
 \end{equation*}
 and similarly for the term $| \langle Y_\mu\cdot 
\eta_L ,  \eta_L\rangle|$. 
 Combining these estimates we get
 \begin{equation}\label{main111}
\begin{array}{rcl}
-C_9  L^\beta d^{\beta(\alpha-1)}& \le &4(\alpha-1)\alpha^3 L^3 d^{3\alpha-4}\left(1+\frac{2\,L\, \alpha (\alpha-1) d^{\alpha-2}}{\mu}+
  \frac{c_0\, \alpha\, L\, d^2}{\mu \, 4(\alpha-1)}
 \right)
\\
&&  + c_5 \, \alpha^2\, L^2\, d^{2\alpha -3}\\
&& + c_6 \, \alpha^2\, L\, d^{\alpha -2} \\
 && +c_7 \, \alpha^2\,L^2\, d^{2\alpha -2}  \\
 &&+c_8 \, \alpha\, L \,d^{\alpha -1}\\
 &&+ c_9.
 \end{array}
\end{equation}
Using \eqref{newkey}  and $\alpha-1<0$ we rewrite it as
\begin{equation}\label{main112}
\begin{array}{rcl}
-C_9  L^\beta d^{\beta(\alpha-1)}& \le L^3 d^{3\alpha-4}&\left[ 2(\alpha-1)\alpha^3\right.
\\
&&  \ + \,c_5 \, \alpha^2\, L^{-1}\, d^{-\alpha +1}\\
&& \ + \,c_6 \, \alpha^2\, L^{-2}\, d^{-2 \alpha +2} \\
 && \ +\,c_7 \, \alpha^2\,L^{-1} \, d^{-\alpha +2}  \\
 &&\ +\,c_8 \, \alpha\, L^{-2}  \,d^{-2 \alpha +3}\\
 &&\left.\,+ c_9\, L^{-3} d^{-3\alpha+4}\right].
 \end{array}
\end{equation}
We now let $L\to\infty$ and use the fact that $L d^\alpha$ is bounded \eqref{zeroorder} to get 

$$ L^3 d^{3\alpha-4 } \le C_{10} L^\beta d^{\beta\alpha} d^{-\beta}\le (L d^\alpha)^\beta d^{-\beta}\le C_{11} d^{-\beta},$$

which implies the boundedness of $L^3 d^{3\alpha-4+\beta}$ as $L\to\infty$. 

Since $d\to0$ by  \eqref{referee} we obtain the desired  contradiction whenever $3\alpha-4-\beta<0$. 
  \addtocontents{toc}{\protect\thispagestyle{empty}}
\pagenumbering{gobble}
\setcounter{page}{1}
\bibliographystyle{alpha}
\bibliography{FerrariManfredi2020}
 \end{document}